\documentclass[12pt,a4paper]{amsart}
\usepackage{amssymb,amsmath}
\usepackage{hyperref}

\numberwithin{equation}{section}
\newtheorem{theorem}{Theorem}[section]
\newtheorem{lemma}[theorem]{Lemma}
\newtheorem{proposition}[theorem]{Proposition}
\newtheorem{corollary}[theorem]{Corollary}

\theoremstyle{definition}
\newtheorem{definition}[theorem]{Definition}
\newtheorem{example}[theorem]{Example}
\newtheorem{conjecture}[theorem]{Conjecture}
\newtheorem{problem}[theorem]{Problem}

\newcommand\Supp{\operatorname{Supp}}
\newcommand\Ass{\operatorname{Ass}}

\newcommand\Ann{\operatorname{Ann}}

\newcommand\Hom{\operatorname{Hom}}
\newcommand\RHom{\operatorname{R Hom}}
\newcommand\Ext{\operatorname{Ext}}
\newcommand\Rad{\operatorname{Rad}}
\newcommand\Ker{\operatorname{Ker}}

\newcommand{\yy}{\underline y}
\newcommand\cd{\operatorname{cd}}

\newcommand\height{\operatorname{height}}
\newcommand\grade{\operatorname{grade}}
\newcommand\ara{\operatorname{ara}}
\newcommand\Spec{\operatorname{Spec}}
\newcommand{\gam}{\Gamma_{I}}
\newcommand{\Rgam}{{\rm R} \Gamma_{\mathfrak m}}

\newcommand{\qism}{\stackrel{\sim}{\longrightarrow}}

\begin{document}
\author[M. Hellus \and P. Schenzel]{Michael Hellus \quad Peter Schenzel}
\title[Cohomologically complete intersections]{On cohomologically
complete intersections }

\address{Universit\"at Leipzig, Fakult\"at f\"ur Mathematik und Informatik,
D --- 04009 Leipzig, Germany}

\email{michael.hellus@math.uni-leipzig.de}

\address{Martin-Luther-Universit\"at Halle-Wittenberg,
Institut f\"ur Informatik, D --- 06 099 Halle (Saale),
Germany}

\email{peter.schenzel@informatik.uni-halle.de}



\begin{abstract}
An ideal $I$ of a local Gorenstein ring $(R, \mathfrak m)$ is called
cohomologically complete intersection whenever $H^i_I(R) = 0$ for
all $i \not= \height I.$ Here $H^i_I(R), i \in \mathbb Z,$ denotes
the local cohomology of $R$ with respect to $I.$ For instance, a
set-theoretic complete intersection is a cohomologically complete
intersection. Here we study cohomologically complete intersections
from various homological points of view, in particular in terms of
their Bass numbers of $H^c_I(R), c = \height I.$ As a main result it
is shown that the vanishing $H^i_I(R) = 0$ for all $i \not= c$ is
completely encoded in homological properties of $H^c_I(R),$ in
particular in its Bass numbers.
\end{abstract}

\maketitle

\section*{Introduction}

Let $(R, \mathfrak m)$ denote a local Noetherian ring. For an ideal
$I \subset R$ it is a rather difficult question to determine the smallest
number $n \in \mathbb N$ of elements $a_1,\ldots, a_n \in R$ such that $\Rad I =
\Rad (a_1,\ldots, a_n)R.$ This number is called the arithmetic rank, $\ara I,$
of $I.$ By Krull's generalized principal ideal theorem it follows that
$\ara I \geq \height I.$ Of a particular interest is the case whenever $\ara I = \height I.$
In this situation $I$ is called a set-theoretic complete intersection.

For the ideal $I$ let $H^i_I(\cdot), i \in \mathbb Z,$ denote the local cohomology
functor with respect to $I,$ see \cite{aG} for its definition and basic results. The
cohomological dimension, $\cd I,$ defined by
\[
\cd I = \sup     \{ i \in \mathbb Z | H^i_I(R) \not= 0\}
\]
is another invariant related to the ideal $I.$ It is well known that $\height I \leq
\cd I \leq \ara I.$ In particular, if $I$ is set-theoretically a complete intersection
it follows that $\height I = \cd I,$ while the converse does not hold in general. Not
so much is known about ideals with the property of $\height I = \cd I.$ We call those
ideals cohomologically complete intersections. In this paper we start with
the investigations of cohomologically complete intersections, in particular when
$I$ is an ideal in a Gorenstein ring $(R, \mathfrak m).$

As an application of our main results there is a characterization of
cohomologically complete intersections for a certain class of ideals
(cf. Theorem \ref{0.1}). A generalization to arbitrary ideals in a
Gorenstein ring is shown in Section 3, namely: $I$ is
cohomologically a complete intersection if and only if a minimal
injective resolution of $H^c_I(R), c = \height I,$ "looks like that
of a Gorenstein ring" (see Theorem \ref{4.2} for the precise
statement).

\begin{theorem} \label{0.1} Let $I$ denote an ideal of a Gorenstein ring
$(R,\mathfrak m)$ with $c = \height I$ and $d = \dim R/I.$ Suppose that $I$
is a complete intersection in $V(I) \setminus \{\mathfrak m\}.$ Then the following conditions
are equivalent:
\begin{enumerate}
\item[(i)] $H^i_I(R) = 0$ for all $i \not= c,$ i.e. $I$ is cohomologically a
complete intersection.
\item[(ii)] $H^d_{\mathfrak m}(H^c_I(R)) \simeq E$ and $H^i_{\mathfrak
m}(H^c_I(R)) = 0,$ for all $i \not=d,$ where $E$ denotes the injective hull of
the residue field $R/\mathfrak m.$
\item[(iii)] The natural map $\Ext^d_R(k, H^c_I(R)) \to k$ is isomorphic and
$\Ext^i_R(k, H^c_I(R)) = 0$ for all $i \not= d.$
\item[(iv)] The Bass numbers of $H^c_I(R))$ satisfy
\[
 \dim_k \Ext^i_R(k, H^c_I(R)) = \delta_{d,i}.
\]
\end{enumerate}
Moreover, if $I$ satisfies the above conditions it follows that
${\hat R}^I \simeq \Hom_R(H^c_I(R),H^c_I(R))$ and $\Ext^i_R(H^c_I(R),H^c_I(R)) = 0$
for all $i \not= 0,$ where ${\hat R}^I$ denotes the $I$-adic completion of $R.$
\end{theorem}

It is a surprising fact -- at least to the authors -- that the
information for the equality $\height I = \cd I$ is completely
encoded in the cohomology module $H^c_I(R).$ Moreover the
characterization of cohomologically complete intersections looks in
a certain sense Gorenstein-like, cf. the well-known cohomological
characterization of Gorenstein rings in terms of Bass numbers and
Theorem \ref{0.1} (iv). The ``strong'' characterization (cf. Theorem
\ref{0.1} (iv)) looks very similar to the ``weak'' characterization
(cf. Theorem \ref{0.1} (iii)). In fact it requires much more effort
(it uses Matlis duals of local cohomology modules \cite{mH}). It is
related to a Conjecture (cf. \ref{3.7}) that the map $\Ext^d_R(k,
H^c_I(R)) \to k$ is in general non-zero.

\section{Preliminaries}
Let $(R, \mathfrak m, k)$ denote a local Gorenstein ring with $n = \dim R.$
In the following let $E = E_R(R/\mathfrak m)$ denote the injective hull of the
residue field $k = R/\mathfrak m$ as $R$-module.

A basic tool for the local cohomology with support in $\{\mathfrak m\}$ is the
local duality theorem (cf. \cite{aG}). To this end let $\hat R$ denote the
completion of $R$ and $\hat M$ the completion of $M.$

\begin{proposition} \label{2.1}
For a finitely generated $R$-module $M$ and an integer
$i \in \mathbb Z$ there are the following natural isomorphisms:
\begin{enumerate}
\item[(a)]
$H^i_{\mathfrak m}(M) \simeq \Hom_R(\Ext_R^{n-i}(M,R), E),$
\item[(b)]
$\Ext_{\hat R}^{n-i}(\hat{M}, \hat{R}) \simeq \Hom_R(H^i_{\mathfrak m}(M),E).$
\end{enumerate}
\end{proposition}

In particular, $H^i_{\mathfrak m}(M)$ is an Artinian $R$-module as a
consequence of Matlis duality. Moreover we have the following
result.

\begin{lemma} \label{2.2} Let $I$ denote an ideal of height $c$ in the Gorenstein
ring $(R, \mathfrak m)$ and $d = \dim R/I = n -c.$ Then the following is true:
\begin{itemize}
\item[(a)] For all $i, j \in \mathbb Z$ there are isomorphisms
\[
\varprojlim \Ext^i_R(\Ext^j_R(R/I^{\alpha}, R),R) \simeq
\Ext_R^i(H^j_I(R),R),
\]
where the inverse system on the left is induced by the
natural projections.
\item[(b)] For all $i, j \in \mathbb Z$ there are natural isomorphisms
\[
\Ext_{\hat R}^{n-i}(H^j_{I \hat R}(\hat R), \hat{R}) \simeq \Hom_R(H^i_{\mathfrak m}(H^j_I(R)),
E),
\]
where $\hat R$ denotes the completion of $R.$
\item[(c)] $H^i_{\mathfrak m}(H^j_I(R)) = 0$ for all integers  $i >
n - j$ and $\Ext^i_R(H^j_I(R),R) = 0$ for all $i <j.$
\end{itemize}
\end{lemma}

\begin{proof} First of all let us recall that the $\Ext$-functors in the first
variable transform a direct system into an inverse system such that
\[
\varprojlim \Ext^i_R(M_{\alpha}, N) \simeq
\Ext^i_R(\varinjlim  M_{\alpha}, N)
\]
for any $R$-module $N$ (cf. \cite{cW}). Now there is the isomorphism $H^j_I(R) \simeq \varinjlim
\Ext^j_R(R/I^{\alpha},R)$ (cf. \cite{aG}). Therefore the first statement is
shown to be true.

For the proof of the second statement first recall that
\[
H^i_{\mathfrak m}(H^j_I(R)) \simeq \varinjlim H^i_{\mathfrak m}(\Ext^j_R(R/I^{\alpha},R))
\]
since the local cohomology commutes with direct limits. Therefore the statement
of (b) is a consequence of the local duality theorem and the previous
observation on the inverse limit.

For the proof of the statement in (c) first recall that $c = \height
I = \height I\hat R.$ Therefore, because of Matlis duality  it will
be enough to prove that $\Ext^i_R(H^j_I(R),R) = 0$ for all integers
$i < j.$ It is a well-known fact that for any $\alpha \in \mathbb N$
we have $\Ext^j_R(R/I^{\alpha},R) = 0$ for all $j < c$ and $j > n.$
Moreover $\dim  \Ext^c_R(R/I^{\alpha},R) = n-d$ and $\dim
\Ext^j_R(R/I^{\alpha},R) \leq n - j$ for all $c < j \leq n$ (cf.
\cite[Proposition 2.3]{pS2}).

Let $\mathfrak b_j = \Ann \Ext^j_R(R/I^{\alpha},R).$ Therefore $\dim R/\mathfrak b_j
\leq n- j$ and $\grade \mathfrak b_j \geq j$ as easily seen.
Therefore $\Ext_R^i(N, R)$ vanishes for
all $R$-modules $N$ such that $\Supp_R N \subseteq V(\mathfrak b_j)$ and all $i < j$
(cf. \cite[Proposition 3.3]{aG}). Because of the statement in (a) this proves the vanishing of
$\Ext^i_R(H^j_I(R),R)$ for all $i< j$ and a given integer $j.$
\end{proof}

As a technical tool in the next section we need a proposition on the behavior of the section functor on certain complexes of $R$-modules.

\begin{proposition} \label{2.3} Let $I$ be an ideal of a Noetherian ring $R.$ Let $X$ denote
 an arbitrary $R$-module with $\Supp_R X \subset V(I).$ Then there is an isomorphism
\[
 \Hom_R(X, \Gamma_I(J^{\cdot})) \simeq \Hom_R(X, J^{\cdot})
\]
for any bounded complex $J^{\cdot}$ of injective $R$-modules.
\end{proposition}

\begin{proof}
 Let $M$ denote a finitely generated $R$-module. Then there is an isomorphism of complexes
\[
 \Hom_R(M, \Gamma_I(J^{\cdot})) \simeq \Gamma_I(\Hom_R(M, J^{\cdot}))
\]
(cf. \cite{rH}). Now let $X$ be an arbitrary $R$-module with $\Supp_R X \subset V(I).$ Then $X\simeq \varinjlim X_{\alpha}$ for a certain direct family of finitely generated submodules $X_{\alpha} \subset X.$ Therefore $\Supp_R X_{\alpha} \subset V(I)$ and
\[
 \Hom_R(X, \Gamma_I(J^{\cdot})) \simeq \varprojlim \Hom_R(X_{\alpha}, \Gamma_I(J^{\cdot})).
\]
Because of the above result on finitely generated $R$-modules it
implies
\[
 \Hom_R(X_{\alpha},\Gamma_I(J^{\cdot})) \simeq \Hom_R(X_{\alpha},  J^{\cdot})
\]
because any module in $\Hom_R(X_{\alpha},  J^{\cdot})^i, i \in \mathbb Z,$ has its support in $V(I).$ By passing to the limit this proves the claim.
\end{proof}

As another technical tool we need a characterization of the
vanishing of the Bass numbers of an arbitrary $R$-module $X.$ Here
$(R, \mathfrak m)$ denotes a local ring with $k = R/\mathfrak m$ its
residue field. Recall that the Bass numbers of $X$ with respect to
$\mathfrak m$ are defined by
\[
 \dim_k \Ext^i_R(k, X),  \text{ for all }  i \in \mathbb Z.
\]

\begin{proposition} \label{2.4} Let $(R,\mathfrak m)$ denote a local ring. Let $X$
 denote an arbitrary $R$-module. For an integer $s \in \mathbb N$ the following conditions are equivalent:
\begin{itemize}
 \item[(i)] $H^i_{\mathfrak m}(X) = 0$ for all $i < s.$
 \item[(ii)] $\Ext^i_R(k, X) = 0$ for all $i < s.$
\end{itemize}
If one of these equivalent conditions is satisfied there is an isomorphism
\[
 \Ext^s_R(k, X) \simeq \Hom_R(k, H^s_{\mathfrak m}(X)).
\]
\end{proposition}

\begin{proof}
We proceed by an induction on $s.$ If $s=0,$ then the equivalence is trivially true.
Moreover, there is an isomorphism $\Hom_R(k, X) \simeq \Hom_R(k, H^0_{\mathfrak m}(X)),$ as easily seen.

So let $s > 0$ and suppose that the claim is true for $s - 1.$ Then there is an isomorphism
\[
 \Ext^{s-1}_R(k, X) \simeq \Hom_R(k, H^{s-1}_{\mathfrak m}(X)).
\]
Because of $\Supp_R H^{s-1}_{\mathfrak m}(X) \subseteq V(\mathfrak m)$ it follows that
$\Ext_R^{s-1}(k, X) = 0$ if and only if $H^{s-1}_{\mathfrak m}(X) = 0.$

In order to complete the inductive step we have to prove the isomorphism of the
statement. To this end consider the spectral sequence
\[
E^{p,q}_2 = \Ext_R^p(k, H^q_{\mathfrak m}(X)) \Longrightarrow E^{p+q}_{\infty} =
\Ext_R^{p+q}(k,X).
\]
Because of $H^i_{\mathfrak m}(X) = 0$ for all $i < s $ and
$\Ext^i_R(k, X) = 0$ for all $i < s$ it degenerates to the following isomorphisms
\[
\Hom_R(k,H^s_{\mathfrak m}(X))=E^{0,s}_2 \simeq E^{0,s}_\infty
\simeq E^s_\infty =\Ext_R^s(k,X).
\]
This completes the proof.
\end{proof}

In the core of the paper we are interested in cohomologically complete intersections $I$ of a
Gorenstein ring $(R, \mathfrak m).$ A technical necessary condition gives the following result.

\begin{proposition} \label{2.5} Suppose that $I$ denotes an ideal of a Gorenstein $(R,\mathfrak m)$
 such that $H^i_I(R) = 0$ for all $i \not= c, c = \height I.$ Then $R/I$ is unmixed, i.e. $c = \height IR_{\mathfrak p}$ for all minimal prime ideals $\mathfrak p \in V(I).$
\end{proposition}

\begin{proof} Suppose there exists a prime ideal $\mathfrak p$ minimal in $V(I)$ such that
 $h = \height IR_{\mathfrak p} > c.$ Then $IR_{\mathfrak p}$ is $\mathfrak p R_{\mathfrak p}$-primary and therefore
\[
 H^h_I(R) \otimes_R R_{\mathfrak p} \simeq H^h_{I R_{\mathfrak p}}(R_{\mathfrak p}) \not= 0.
\]
But this means that $H^h_I(R) \not= 0, h > c,$ in contradiction to the assumption.
\end{proof}

As a further technical tool we shall prove a result for complexes which is well-known in the case
of a module.

\begin{proposition} \label{2.7} Let $R$ denote a Noetherian ring and $I \subset R$ an ideal. Let $X^{\cdot}$ denote a complex of $R$-modules such that $\Supp H^i(X^{\cdot}) \subseteq V(I)$ for all $i \in \mathbb Z.$ Then $H^i_I(X^{\cdot}) \simeq H^i(X^{\cdot})$ for all $i \in \mathbb Z.$
\end{proposition}

\begin{proof}
In order to compute the hypercohomology of the complex ${\rm R}\Gamma_I(X^{\cdot})$ there is the following spectral sequence
\[
 E_2^{p,q} = H^p_I(H^q(X^{\cdot})) \Longrightarrow E^{p+q}_{\infty} =
H^{p+q}_I(X^{\cdot})).
\]
By the assumption $\Supp H^i(X^{\cdot}) \subseteq V(I)$ so that $E_2^{p,q} = 0$ for all $p \not= 0.$
This implies a partial degeneration of the spectral sequence to the isomorphisms
\[
 H^i_I(X^{\cdot}) \simeq H^0_I(H^i(X^{\cdot})) \simeq H^i(X^{\cdot})
\]
for all $i \in \mathbb Z.$
\end{proof}

As another auxiliary tool in the following we need the so-called
Hartshorne-Lichtenbaum vanishing theorem. We give it here for a
slight sharpened form as we shall need it. A proof can e.~g. be
found in \cite[8.1.9 and 8.2.1]{mH}.

\begin{proposition} \label{2.8} Let $(R,\mathfrak m)$ denote a local Gorenstein ring with $n = \dim R.$ Let $I$ be an ideal of $R.$ Let $M$ be a finitely generated $R$-module. Then
 \[
  \Ass_{\hat{R}} \Hom_R(H^n_I(M), E) = \{ \mathfrak p \in \Ass_{\hat{R}} \hat{M}  | \dim \hat{R}/\mathfrak p = n \text{ and } \dim \hat{R}/(\mathfrak p + I\hat{R}) = 0\}.
 \]
As an epimorphic image of $H^n_{\mathfrak m}(M)$, $H^n_I(M)$ is an
Artinian $R$-module. Here $\hat{M}$ denotes the completion of $M.$
\end{proposition}

\section{The truncation complex}
In this section let $(R, \mathfrak m)$ denote a local Gorenstein ring and $d = \dim R.$ Let $R \qism E^{\cdot}$ denote a minimal injective resolution of $R$ as an $R$-module. It is a well-known fact that
\[
E^i \simeq \oplus_{\mathfrak p \in \Spec R, \height \mathfrak p = i} E_R(R/\mathfrak p),
\]
where $E_R(R/\mathfrak p)$ denotes the injective hull of $R/\mathfrak p.$  See \cite{hB} for  these and related results on Gorenstein rings.

Now let $I \subset R$ denote an ideal and $c = \height I.$ Then $d = \dim R/I = n - c.$ The local cohomology modules $H^i_I(R), i \in \mathbb Z,$ are -- by definition -- the cohomology modules of the complex $\gam (E^{\cdot}).$ Because of $\gam(E_R(R/\mathfrak p)) = 0$ for all
$\mathfrak p \not\in V(I)$ it follows that $\gam(E^{\cdot})^i = 0$ for all $i < c.$

Therefore $H^c_I(R) = \Ker(\gam(E^{\cdot})^c \to \gam(E^{\cdot})^{c+1}).$ This observation provides an embedding $H^c_I(R)[-c] \to \gam(E^{\cdot})$ of complexes of $R$-modules.

\begin{definition} \label{3.1} The cokernel of the embedding  $H^c_I(R)[-c] \to \gam (E^{\cdot})$ is defined as $C^{\cdot}_R(I),$ the truncation complex. So there is a short exact sequence of complexes of $R$-modules
\[
0 \to H^c_I(R)[-c] \to \gam (E^{\cdot}) \to C^{\cdot}_R(I) \to 0.
\]
In particular it follows that $H^i(C^{\cdot}_R(I)) = 0$ for $i \leq c$ or $i > d$ and $H^i(C^{\cdot}_R(I)) \simeq H^i_I(R)$ for $c < i \leq d.$
\end{definition}

The advantage of the truncation complex is that it separates information of the local cohomology modules $H^i_I(R), i = c,$ from those with $i \not= c.$ A first result in this direction is the following lemma.

\begin{lemma} \label{3.2} With the previous notation there are an exact sequence
 \[
0 \to H^{n-1}_{\mathfrak m}(C^{\cdot}_R(I)) \to H^d_{\mathfrak m}(H^c_I(R)) \to E \to H^n_{\mathfrak m}(C^{\cdot}_R(I)) \to 0,
 \]
isomorphisms $H^{i-c}_{\mathfrak m}(H^c_I(R)) \simeq H^{i-1}_{\mathfrak m}(C^{\cdot}_R(I))$ for $i < n$ and the vanishing $H^{i-c}_{\mathfrak m}(H^c_I(R)) = 0$ for $i > n.$
\end{lemma}

\begin{proof}
Take the short exact sequence of the truncation complex (cf. \ref{3.1}) and apply the derived functor $\Rgam (\cdot).$ In the derived category this provides a short exact sequence of complexes
\[
0 \to \Rgam (H^c_I(R))[-c] \to \Rgam (\gam(E^{\cdot})) \to \Rgam (C^{\cdot}_R(I)) \to 0.
\]
Since $\gam(E^{\cdot})$ is a complex of injective $R$-modules we might use $\Gamma_{\mathfrak m}(\Gamma_I(E^{\cdot}))$ as a representative of $\Rgam (\gam(E^{\cdot})).$ But now there is an equality for the composite of section functors $\Gamma_{\mathfrak m}(\Gamma_I(\cdot)) = \Gamma_{\mathfrak m}(\cdot).$ Therefore $\Gamma_{\mathfrak m}(E^{\cdot})$ is a representative of $\Rgam (\gam(E^{\cdot}))$ in the derived category. But now $\Gamma_{\mathfrak m}(E_R(R/\mathfrak p)) = 0$ for any prime ideal $\mathfrak p \not= \mathfrak m$ while $\Gamma_{\mathfrak m}(E) \simeq E.$ So there is an isomorphism of complexes $\Gamma_{\mathfrak m}(E^{\cdot}) \simeq E[-n].$

With these observations in mind the above short exact sequence induces the exact sequence of the statement and the isomorphisms
\[
H^{i-c}_{\mathfrak m}(H^c_I(R)) \simeq H^{i-1}_{\mathfrak m}(C^{\cdot}_R(I)), \quad i < n
\]
by view of the corresponding long exact cohomology sequence. Moreover the vanishing of
$H^i_{\mathfrak m}(H^c_I(R))$ for all $i > d$ is shown above (cf. \ref{2.2}).
\end{proof}

As a consequence there is the following necessary condition for an
ideal $I \subset R$ to be a cohomologically complete intersection.
As we shall see later this is not sufficient (cf. \ref{5.1}).

\begin{corollary} \label{3.3} Let $I\subset R$ be an ideal with $\height I = c.$ Suppose that
$H^i_I(R) = 0$ for all $i \not= c.$ Then $H^d_{\mathfrak m}(H^c_I(R)) \simeq E$ and
$H^i_{\mathfrak m}(H^c_I(R)) = 0$ for all $i \not= c.$
\end{corollary}

\begin{proof} By the assumption we have the vanishing of $H^i_I(R)$ for all $i \not= c.$
Therefore the truncation complex $C^{\cdot}_R(I)$ is bounded and homologically trivial.
In order to compute the hypercohomology $H^i_{\mathfrak m}(C^{\cdot}_R(I))$ consider the
following spectral sequence
\[
E_2^{p,q} = H^p_{\mathfrak m}(H^q(C^{\cdot}_R(I))) \Longrightarrow E^{p+q}_{\infty} =
H^{p+q}_{\mathfrak m}(C^{\cdot}_R(I)).
\]
Because all the initial terms vanish and because of the finiteness of the spectral sequence
$H^i_{\mathfrak m}(C^{\cdot}_R(I)) = 0$ for all $i \in \mathbb Z.$ So the claim is true by
Lemma \ref{3.2}.
\end{proof}

Let us continue with a result -- in a certain sense -- dual to the statement of Lemma \ref{3.2}. Here let
${\hat R}^I$ denote the $I$-adic completion of $R.$

\begin{lemma} \label{3.4} With the above notation the following is true:
 \begin{enumerate}
 \item[(a)] There is a short exact sequence
\[
0 \to \Ext^0_R(C^{\cdot}_R(I),R) \to {\hat R}^I \to \Ext^c_R(H^c_I(R),R) \to \Ext^1_R(C^{\cdot}_R(I),R) \to 0.
\]
\item[(b)] There are isomorphisms $\Ext^{i+c}_R(H^c_I(R),R) \simeq \Ext^{i+1}_R(C^{\cdot}_R(I),R)$ for $i > 0.$
\item[(c)] We have the vanishing $\Ext^i_R(C^{\cdot}_R(I),R) = 0$ for $i < 0.$
\end{enumerate}
\end{lemma}

\begin{proof} Let $R \qism E^{\cdot}$ denote the minimal injective resolution of the
Gorenstein ring $R.$ Apply the functor $\Hom(\cdot, E^{\cdot})$ to the
short exact sequence of the truncation complex as defined in
\ref{3.1}. Since $E^{\cdot}$ is a complex of injective $R$-modules it
provides a short exact sequence of complexes
\[
0 \to \Hom_R(C^{\cdot}_R(I), E^{\cdot}) \to \Hom_R(\gam(E^{\cdot}),E^{\cdot}) \to \Hom(H^c_I(R),E^{\cdot})[c] \to 0.
\]
By the definition the cohomology modules of the complexes on the left resp. on the right are 
$\Ext^i_R(C^{\cdot}_R(I), R)$ resp. $\Ext^{i-c}_R(H^c_I(R), R)$ for all $i \in \mathbb Z.$

Let us consider the complex in the middle. Let $\yy = y_1, \ldots, y_s$ be a generating set
of the ideal $I.$ Then $\gam(E^{\cdot}) \qism \Check{C}_{\yy} \otimes E^{\cdot},$ where $\Check{C}_{\yy}$ 
denotes the \v{C}ech complex with respect to $\yy$ (cf \cite[Theorem 1.1]{pS} for the details).

Therefore $\Hom_R(\gam(E^{\cdot}), E^{\cdot}) \simeq \Hom_R(\Check{C}_{\yy},\Hom_R(E^{\cdot}, E^{\cdot})).$ Because 
of the quasi-isomor\-phism $R \qism \Hom_R(E^{\cdot}, E^{\cdot}),$ recall that $ E^{\cdot}$
is a dualizing complex of $R.$ Therefore the last complex is a representative of 
$\operatorname{R} \Hom_R(\Check{C}_{\yy}, R)$ in the derived category. But it follows that
\[
 \operatorname{R} \Hom_R(\Check{C}_{\yy}, R) \qism {\hat R}^I,
\]
as shown in \cite[Theorem 1.1]{pS}. Therefore with the above considerations the long exact 
cohomology sequence provides the statements (a) and (b) of the claim. With the aid of \ref{2.2} 
this proves also the statement in (c).
\end{proof}

As an application there is another necessary criterion for an ideal $I \subset R$ to
be cohomologically a complete intersection.

\begin{corollary} \label{3.5} Let $I \subset R$ denote an ideal of $\height I = c.$ Suppose that
 $H^i_I(R) = 0$ for all $i \not= c.$ Then there is an isomorphism ${\hat R}^I \simeq \Ext^c_R(H^c_I(R),R)$
and the vanishing $\Ext^{i+c}_R(H^c_I(R),R) = 0$ for all $i \not= 0.$
\end{corollary}

\begin{proof} By virtue of the Lemma \ref{3.4} It will be enough to prove that $\Ext^i_R(C^{\cdot}_R(I),R)$ vanishes for all $i \not= 0.$ Let $R \qism E^{\cdot}$ denote the
 minimal injective resolution of $R.$ Then $\Hom_R(C^{\cdot}_R(I),E^{\cdot})$ is a representative
of $\Ext^i_R(C^{\cdot}_R(I),R)$ in the derived category. By the assumption $C^{\cdot}_R(I)$ is an
exact complex. So the conclusion follows because $\Hom_R(\cdot, E^{\cdot})$ preserves exactness.
\end{proof}

In the next result there is a consideration of the Bass numbers of the local cohomology modules $H^c_I(R).$ This provides a certain necessary numerical condition for $I$ to be a cohomologically complete intersection as we shall see later.

\begin{lemma} \label{3.6} With the above notation there is a natural homomorphism
 \[
 \phi :  \Ext^d_R(k, H^c_I(R)) \to k  .
 \]
In addition, suppose that $H^i_I(R) = 0$ for all $i \not= c.$ Then $\phi$ is an isomorphism and
$ \Ext^i_R(k, H^c_I(R)) = 0$ for all $i \not= d.$
\end{lemma}

\begin{proof}
 Apply the derived functor $\operatorname{R} \Hom_R(k,\cdot)$ to the short exact sequence as it is defined in the definition of the truncation complex (cf. \ref{3.1}). Then there is the following short
exact sequence of complexes in the derived category
\[
 0 \to \operatorname{R} \Hom_R(k,H^c_I(R))[-c] \to \operatorname{R} \Hom_R(k,\Gamma_{\mathfrak m}(E^{\cdot})) \to \operatorname{R} \Hom_R(k,C^{\cdot}_R(I)) \to 0.
\]
Now we consider the complex in the middle. It is represented by $\Hom_R(k, \Gamma_{\mathfrak m}(E^{\cdot}))$ since $\Gamma_{\mathfrak m}(E^{\cdot})$ is a complex of injective modules.
Moreover there are the following isomorphisms
\[
 \Hom_R(k, \Gamma_{\mathfrak m}(E^{\cdot})) \simeq \Hom_R(k, E^{\cdot}) \simeq E[-n].
\]
By virtue of the long exact cohomology sequence it yields the natural homomorphism of the statement. Under the additional assumption of $H^i_I(R) = 0$ for all $i \not= c$ it follows that $C^{\cdot}_R(I)$ is
an exact complex. Therefore $\operatorname{R} \Hom_R(k,C^{\cdot}_R(I))$ is also an exact complex.
Whence the long exact cohomology sequence applied to the above short exact sequence provides
the statements on the Bass numbers of $H^c_I(R).$
\end{proof}

\begin{conjecture} \label{3.7}
It is an open problem whether the natural homomorphism
\[
\phi : \Ext^d_R(k, H^c_I(R)) \to k
\]
is in general non-zero, that is a surjection. In general the $k$-vector space $\Ext^d_R(k, H^c_I(R))$ is
not of finite dimension (cf. Example \ref{5.2}).
\end{conjecture}

In the next we are interested in the endomorphism ring of $H^c_I(R).$ As a consequence it
provides another necessary condition for an ideal $I$ to be a cohomologically complete
intersection.

\begin{lemma} \label{3.8}
 Let $I \subset R$ denote an ideal of a Gorenstein ring $(R,\mathfrak m)$ and $c = \height I.$ Then there is a natural isomorphism
\[
 \Hom_R(H^c_I(R), H^c_I(R)) \simeq \Ext^c_R(H^c_I(R),R).
\]
Moreover, suppose that $H^i_I(R) = 0$ for all $i \not= c.$ Then ${\hat R}^I \simeq \Hom_R(H^c_I(R), H^c_I(R))$ and $\Ext^i_R(H^c_I(R),H^c_I(R)) = 0$ for all $i \not= 0.$
\end{lemma}

\begin{proof}
Consider the short exact sequence of complexes as introduced by the definition of the truncation complex  in \ref{3.1}. Apply the derived functor $\operatorname{R} \Hom_R(H^c_I(R), \cdot)$ to this sequence. So, in the derived category there is a short exact sequence of complexes
\[
\begin{gathered}
0 \to \RHom_R(H^c_I(R), H^c_I(R))[-c] \to \RHom_R(H^c_I(R), \Gamma_I(E)) \to \\
\to \RHom_R(H^c_I(R), C^{\cdot}_R(I)) \to 0.
\end{gathered}
\]
Because $\Gamma_I(E^{\cdot})$ is a complex of injective $R$-modules a representative for the
complex in the middle is given by $\Hom_R(H^c_I(R), \Gamma_I(E^{\cdot})).$ Now $\Supp_R(H^c_I(R)) \subset V(I).$ Whence Proposition \ref{2.3} shows that this complex is isomorphic to $\Hom_R(H^c_I(R),E^{\cdot}).$
So the above short exact sequence of complexes induces an exact sequence
\[
\begin{gathered}
\Ext_R^{c-1}(H^c_I(R), C_R(I)) \to \Hom_R(H^c_I(R),H^c_I(R)) \to \\ \to \Ext_R^c( H^c_I(R),
R) \to \Ext_R^c(H^c_I(R),C^{\cdot}_R(I)).
\end{gathered}
\]
In order to finish the proof of the first statement it will be enough to show that
$\Ext^i_R(H^c_I(R),C^{\cdot}_R(I)) = 0$ for $i = c-1, c.$

To this end let $C^{\cdot}_R(I) \qism F^{\cdot}$ denote an injective resolution
of the complex $C^{\cdot}_R(I)$ (cf. \cite{rH}). Then by definition it
follows that $H^i(F^{\cdot}) = 0$ for all $i \leq c$ and all $i > n$ resp.
$H^i(F^{\cdot}) \simeq H^i_I(R)$ for $c < i \leq n.$ Moreover,
\[
\Ext^i_R(H^c_I(R),C^{\cdot}_R(I)) \simeq H^i(\Hom_R(H^c_I(R),F^{\cdot})).
\]
In order to compute the cohomology of the complex $\Hom_R(H^c_I(R),F^{\cdot})$
there is the following spectral sequence
\[
E^{p,q}_2 = \Ext_R^p(H^c_I(R), H^q(F^{\cdot})) \Longrightarrow E^{p+q}_{\infty} =
H^{p+q}(\Hom_R(H^c_I(R), F^{\cdot})).
\]
Let $p+q \leq c.$ For $p \geq 0$ it follows that $q \leq c.$ Therefore it turns
out that
\[
\Ext_R^p(H^c_I(R), H^q(F^{\cdot}))) = 0 \text{ for all } p+q \leq c.
\]
So that $H^i(\Hom_R(H^c_I(R),F^{\cdot})) = 0$ for all $i \leq c$ as a consequence of the spectral sequence.
This proves the first isomorphism of the statement.

Now assume that $H^i_I(R) = 0$ for all $i \not= c.$ Then the complex
$\Hom_R(H^c_I(R),F^{\cdot})$ is exact. So there are isomorphisms
\[
\Ext^{i-c}_R(H^c_I(R),H^c_I(R)) \simeq  \Ext^i_R(H^c_I(R), R) \text{ for all } i \in \mathbb Z.
\]
By view of Corollary \ref{3.5} this completes the proof.
\end{proof}

The previous result is a slight extension of results of the first
author and St\"uckrad (cf. \cite[2.2 (iii)]{HS}). There it is shown
that the endomorphism ring of $H^c_I(R)$ is isomorphic to $R$ for a
cohomologically complete intersection ideal $I$ in a complete local
Gorenstein ring $(R, \mathfrak m).$

Moreover, the previous result has an interesting application. It implies the non-vanishing of a
certain local cohomology module of $H^c_I(R), c = \height I.$

\begin{corollary} \label{3.9} Let $I \subset R$ denote an ideal of a Gorenstein ring $(R,\mathfrak m)$
 and $c = \height I.$ Then $H^d_{\mathfrak m}(H^c_I(R)) \not= 0$ where $d = \dim R/I.$
\end{corollary}

\begin{proof} By view of Lemma \ref{2.2} (b) it will be enough to show that
$\Ext_{\hat R}^c(H^c_{I \hat R}(\hat R), \hat{R})$ does not vanish.
Let us assume that $R$ is a complete local ring. By virtue of
Corollary \ref{3.8} there is the following isomorphism
$\Hom_R(H^c_I(R), H^c_I(R)) \simeq \Ext^c_R(H^c_I(R),R).$ Because of
$H^c_I(R) \not= 0$ the endomorphism ring of $H^c_I(R)$ is
non-trivial.
\end{proof}

\section{Main results}
In this section let $(R, \mathfrak m)$ denote a $n$-dimensional
Gorenstein ring. Let $I \subset R$ be an ideal with $c = \height I$
and $\dim R/I = n -c.$ Then we shall prove our first
characterization of cohomologically complete intersections. To this
end let us fix the abbreviation $h(\mathfrak p) = \dim R_{\mathfrak
p} - c$ for a prime ideal $\mathfrak p \in V(I).$

\begin{theorem} \label{4.1} With the previous notation the following conditions are equivalent:
 \begin{itemize}
  \item[(i)] $H^i_I(R) = 0$ for all $i \not= c$, i.~e. $I$ is a cohomologically complete
  intersection.
  \item[(ii)] For all $\mathfrak p \in V(I)$ the natural map
\[
 H^{h(\mathfrak p)}_{\mathfrak pR_{\mathfrak p}}(H^c_{IR_{\mathfrak p}}(R_{\mathfrak p})) \to E(k(\mathfrak p))
\]
is an isomorphism and $H^i_{\mathfrak pR_{\mathfrak p}}(H^c_{IR_{\mathfrak p}}(R_{\mathfrak p})) = 0$ for all $i \not= h(\mathfrak p).$
   \item[(iii)] For all $\mathfrak p \in V(I)$ the natural map
\[
 H^{h(\mathfrak p)}_{\mathfrak p\widehat{R_{\mathfrak p}}}(H^c_{I\widehat{R_{\mathfrak p}}}(\widehat{R_{\mathfrak p}})) \to E(k(\mathfrak p))
\]
is an isomorphism and $H^i_{\mathfrak p\widehat{R_{\mathfrak p}}}(H^c_{I\widehat{R_{\mathfrak p}}}(\widehat{R_{\mathfrak p}})) = 0$ for all $i \not= h(\mathfrak p).$
  \item[(iv)] For all $\mathfrak p \in V(I)$ the natural map
\[
 \widehat{R_{\mathfrak p}} \to \Ext^c_{\widehat{R_{\mathfrak p}}}(H^c_{I\widehat{R_{\mathfrak p}}}(\widehat{R_{\mathfrak p}}),\widehat{R_{\mathfrak p}})
\]
is an isomorphism and $\Ext^i_{\widehat{R_{\mathfrak p}}}(H^c_{I\widehat{R_{\mathfrak p}}}(\widehat{R_{\mathfrak p}}),\widehat{R_{\mathfrak p}}) = 0$ or all $i \not= c.$
  \item[(v)] For all $\mathfrak p \in V(I)$ the natural map
\[
 \Ext^{h(\mathfrak p)}_{R_{\mathfrak p}}(k(\mathfrak p), H^c_{I R_{\mathfrak p}}(R_{\mathfrak p})) \to k(\mathfrak p)
\]
is an isomorphism and $\Ext^i_{R_{\mathfrak p}}(k(\mathfrak p), H^c_{I R_{\mathfrak p}}(R_{\mathfrak p})) = 0$  for all $i \not= h(\mathfrak p).$
 \end{itemize}
\end{theorem}

\begin{proof}
(i) $\Rightarrow $ (ii): Let $\mathfrak p \in V(I).$ Then $\height
IR_{\mathfrak p} = c$ (cf. Proposition \ref{2.5}) and $h(\mathfrak
p) = \dim R_{\mathfrak p}/IR_{\mathfrak p}$ as easily seen. Clearly
$H^i_{I R_{\mathfrak p}}(R_{\mathfrak p}) = 0$ for all $i \not= c =
\height IR_{\mathfrak p}.$ Therefore the statement turns out by
virtue of Corollary \ref{3.3}.

(ii) $\Leftrightarrow $ (iii): This equivalence is a consequence of
the faithful flatness of $R_{\mathfrak p} \to \widehat{R_{\mathfrak
p}}$ and the fact that local cohomology commutes with flat
extensions.

(iii) $\Leftrightarrow $ (iv): By Matlis duality this equivalence follows by Lemma \ref{2.2} (b). Recall
that $h(\mathfrak p) = \dim R_{\mathfrak p} -c$ by definition.

(ii) $\Leftrightarrow $ (v): Since both of the conditions localize it will be enough to prove
the equivalence for the maximal ideal of a local Gorenstein ring $(R, \mathfrak m).$ Then the equivalence of the vanishings follow (cf. Proposition \ref{2.4}) for $X = H^c_I(R)$ and all $i < d.$ Moreover it provides that the  natural map $\Ext^d_R(k, H^c_I(R)) \to \Hom_R(k, H^d_{\mathfrak m}(H^c_I(R))$ is an isomorphism. There is a commutative diagram
\[
\begin{array}{ccl}
 \Ext^d_R(k, H^c_I(R)) & \to & k \\
  \downarrow & & \downarrow \\
 \Hom_R(k, H^d_{\mathfrak m}(H^c_I(R))) & \to & k.
\end{array}
\]
By the construction (cf. Lemma \ref{3.2} and Lemma \ref{3.6}) it turns out that the second vertical map is the identity. Therefore $\Ext^d_R(k, H^c_I(R)) \to k$ is an isomorphism if and only if $\Hom_R(k, H^d_{\mathfrak m}(H^c_I(R)))  \to  k$ is an isomorphism. Now suppose that $H^d_{\mathfrak m}(H^c_I(R)) \to E$ is an isomorphism. Then it follows easily that $\Ext^d_R(k, H^c_I(R)) \to k$ is an isomorphism.
The converse follows by Theorem \ref{4.2} in a more general context.

(ii) $\Rightarrow $ (i): We proceed by induction on $d = \dim R/I.$
In the case of $d = 0$ the ideal $I$ is $\mathfrak m$-primary.
Therefore the statement is true because $R$ is a Gorenstein ring. So
let $d > 0.$ By view of Corollary \ref{3.9} $\dim R_{\mathfrak p} -
c = \dim R_{\mathfrak p}/IR_{\mathfrak p}$ and therefore $c =
\height IR_{\mathfrak p}$ for all $\mathfrak p \in V(I).$ By
induction hypothesis it follows that $H^i_{IR_{\mathfrak
p}}(R_{\mathfrak p}) = 0$ for all $i \not= c$ and all $\mathfrak p
\in V(I)\setminus \{ \mathfrak m\} .$ This means that $\Supp
H^i_I(R)) \subset \{\mathfrak m\}$ for all $i \not = c.$ Therefore
\[
 H^i_{\mathfrak m}(C^{\cdot}_R(I)) \simeq H^i(C^{\cdot}_R(I)) \simeq H^i_I(R) \text{ for } c < i \leq n
\]
and $H^i_{\mathfrak m}(C^{\cdot}_R(I)) = 0 $ for $i \leq c$ and $i > n.$ Because of the assumption for
$\mathfrak p = \mathfrak m,$ the maximal ideal, it follows that $H^i_I(R) = 0$ for all $i \not = c$ (cf. Lemma \ref{3.2}).
\end{proof}

Before we shall go into the details of the proof of Theorem \ref{0.1} we have to complete the proof of previous Theorem \ref{4.1} in the light of the Conjecture \ref{3.7}. To be more precise
we shall prove the following important result.

\begin{theorem} \label{4.2}
With the notion from the beginning of this section the following conditions are equivalent:
\begin{enumerate}
 \item[(i)] $H^i_I(R) = 0$ for all $i \not= c$, i.~e. $I$ is a cohomologically complete
  intersection.
 \item[(ii)] For every $\mathfrak p \in V(I)$ it holds
\[
 \dim_{k(\mathfrak p)} \Ext^i_{R_{\mathfrak p}}(k(\mathfrak p), H^c_{IR_{\mathfrak p}}(R_{\mathfrak p})) = \delta_{h(\mathfrak p),i}
\]
for all $i \in \mathbb Z.$
\end{enumerate}
\end{theorem}

\begin{proof}
The implication (i) $\Rightarrow$ (ii) is a consequence of the
previous Theorem \ref{4.1}. In order to prove the reverse
implication (ii) $\Rightarrow$ (i) we proceed by induction on $d =
\dim R/I.$ Because $R$ is a Gorenstein ring  the claim is obviously
true for $d = 0.$ For the next let $d = 1,$ i.e. $n = c + 1.$ Then
$\Supp_R H^n_I(R) \subset \{\mathfrak m\}$ and we have to show
$H^n_I(R) = 0.$ There is the following spectral sequence
\[
E^{p,q}_2=\Ext^p_R(R/\mathfrak m,H^q_I(R))\Rightarrow  E^{p+q}_{\infty} = \Ext^{p+q}_R(R/\mathfrak m,R)
\]
and, as a part of it, the boundary homomorphism
\[
E^{0,n}_2 = \Hom_R(R/\mathfrak m,H^n_I(R)) \stackrel{d}{\to}
E^{2,c}_2 = \Ext^2_R(R/\mathfrak m,H^c_I(R)) = 0.
\]
Moreover there is the isomorphism $\Ext^1_R(R/\mathfrak m, H^c_I(R)) = E^{1,c}_2 \simeq E^{1,c}_{\infty}.$ Recall that all rows except those with $q = c,c+1$ are zero. By the hypothesis this is a one-dimensional
$k$-vector space. Moreover $E^n_{\infty} = \Ext^n_R(k, R)$ is also one-dimensional and therefore $E^{0,n}_{\infty}$ has to be zero. But
\[
 0 = E^{0,n}_{\infty} \simeq E^{0,n}_3 = \Ker d.
\]
and $d$ is injective. But this implies that $H^n_I(R) = 0,$ as required.

Now let $d = \dim R/I > 1.$ By the inductive hypothesis and because
of $c = \height IR_{\mathfrak p}$ for all $\mathfrak p \in V(I)$
(cf. Corollary \ref{3.9}) it follows that $\Supp_R H^i_I(R) \subset
\{\mathfrak m\}$ for all $i \not= c.$ Therefore $H^i_{\mathfrak
m}(C^{\cdot}_R(I)) \simeq H^i_I(R)$ for all $c\leq i \leq n$ and zero
elsewhere (cf. Proposition \ref{2.7}). By virtue of Lemma \ref{3.2}
the assumption implies that $H^i_I(R) = 0$ for all $i \not= n-1, n$
and $i \not= c.$ So it remains to show that $H^i_I(R) = 0$ for $i =
n-1, n.$ Without loss of generality we may assume that $R$ is
complete since $R \to \hat{R}$ is a faithful flat extension and
commutes with local cohomology.

As above let $R\qism E^{\cdot}$ denote a minimal injective resolution of
$R.$ Then define the complex $E^{\cdot}_1 = \Gamma_I(E^{\cdot})[-c].$ By the
previous observation it is up to cohomological degrees $d, d-1$ an
injective resolution of $H^c_I(R).$ Moreover let $E^{\cdot}_2$ denote a
minimal injective resolution of $H^c_I(R).$ The assumption on the
Bass numbers in (ii) provides that
\[
 E^i_2 = \oplus_{\mathfrak p \in V(I), h(\mathfrak p) = i} E_R(R/\mathfrak p),
\]
where $E_R(R/\mathfrak p)$ denotes the injective hull of $R/\mathfrak p.$ Therefore there is
a comparison map of complexes $\phi^{\cdot}: E^{\cdot}_2 \to E^{\cdot}_1$ such that $\phi^i$ is the
identity for all $i \not= d.$ Moreover $E^d_2 = E^d_1 = E(k),$ the injective hull of the residue field. By Matlis duality it follows that the endomorphism $\phi^d$ is given by the multiplication with
a certain element $x\in R.$ It is easily seen that $\phi^{\cdot}$ induces in homological degree $d-1$ and $d$ resp. the following isomorphisms
\[
 H^{n-1}_I(R) \simeq 0:_E x \text{ and } H^n_I(R) \simeq E/xE.
\]
Still we have to show that both vanish. To this end let $D(\cdot) = \Hom_R(\cdot, E)$ denote the
Matlis functor. Then
\[
 D(H^{n-1}_I(R)) \simeq R/xR \text{ and } D(H^n_I(R)) \simeq 0 :_R x.
\]
The associated prime ideals of $R$ are all of dimension $n.$ We split them into two disjoint subsets
\[
 U = \{\mathfrak p \in \Ass R  | \dim R/(I + \mathfrak p) = 0\} \text{ and }
 V = \{\mathfrak p \in \Ass R  | \dim R/(I + \mathfrak p) > 0\}.
\]
First of all $V$ can not be empty. Since otherwise $U = \Ass R$ and $I$ is an $\mathfrak m$-primary ideal in contradiction to $d = \dim R/I > 1.$ Second we claim that $U$ can not be empty. Otherwise it follows by the Hartshorne-Lichtenbaum Theorem (cf. Proposition \ref{2.8}) that $H^n_I(R) = 0.$ Therefore the multiplication by $x$ on $R$ is injective and
\[
 \Ass_R D(H^{n-1}_I(R)) = \Ass R/xR.
\]
In case $xR$ is a proper ideal it implies that $\mathfrak p$ is minimal and of height one for any $\mathfrak p \in \Ass R/xR.$ Therefore
\[
 0 \not= \Hom_R(R/\mathfrak p, D(H^{n-1}_I(R)) \simeq D(H^{n-1}_I(R/\mathfrak p))
\]
for all $\mathfrak p \in \Ass_R R/xR.$ Whence $\dim R/(\mathfrak p + I) = 0$ for all $\mathfrak p \in \Ass_R R/xR$ by the Hartshorne-Lichtenbaum theorem. Because of $\Rad (I + xR) =\mathfrak m$ this implies $d = \dim R/I \leq 1,$ a contradiction.

Now define $\mathfrak a$ and $\mathfrak b$ resp. the intersections of all the primary components of the zero ideal of $R$ where the corresponding primes belong to $U$ and $V$ resp. Because both, $U$ and $V$ are non-empty, $\mathfrak a$ as well as $\mathfrak b$ is a proper ideal in $R$ and $0 = \mathfrak a \cap \mathfrak b.$  The natural short exact sequence
\[
 0 \to R \to R/\mathfrak a \oplus R/\mathfrak b \to R/(\mathfrak a + \mathfrak b) \to 0
\]
induces an exact sequence
\[
 \begin{gathered}
  D(H^{n-1}_I(R/(\mathfrak a + \mathfrak b)) \to D(H^{n-1}_I(R/\mathfrak a))) \oplus D(H^{n-1}_I(R/\mathfrak b)) \to \\
D(H^{n-1}_I(R)) \to D(H^{n-2}_I(R/(\mathfrak a + \mathfrak b))).
 \end{gathered}
\]
In the following we want to compare the maximal members of the supports of the modules in this exact sequence. Because of $\Rad( \mathfrak a + I) = \mathfrak m$ as follows by the definition of $\mathfrak a$ there are the following isomorphisms
\[
\begin{aligned}
 D(H^i_I(R/(\mathfrak a + \mathfrak b)) & \simeq   D(H^i_{\mathfrak m}(R/(\mathfrak a + \mathfrak b)), \\
 D(H^i_I(R/\mathfrak a)) & \simeq  D(H^i_{\mathfrak m}(R/\mathfrak a))
\end{aligned}
\]
for all $i \in \mathbb Z.$ For the finitely generated $R$-modules on the right hand side, the so-called modules of deficiency,
there are estimates of their dimension. In particular
\[
 \dim D(H^i_{\mathfrak m}(R/(\mathfrak a + \mathfrak b))) \leq n -1, i = n-1, n-2, \text{ and }
\dim D(H^{n-1}_{\mathfrak m}(R/\mathfrak a)) \leq n - 1
\]
because $\height (\mathfrak a + \mathfrak b) \geq 1$ (and every
$R/(\mathfrak a+\mathfrak b)$-module has dimension at most $n-1$). These estimates in accordance
with the above short exact sequence provide the following equality
about associated prime ideals
\[
 \Ass R \cap \Ass D(H^{n-1}_I(R)) = \Ass R \cap \Ass D(H^{n-1}_I(R/\mathfrak b)).
\]
In case of  $H^n_I(R) \not= 0$ there is a prime  $\mathfrak q \in \Ass D(H^n_I(R)) = \Ass (0:_R x).$ Therefore
$\dim R/\mathfrak q = n$ and $x \in \mathfrak q.$ This implies that $\mathfrak q \in \Ass R/xR = \Ass D(H^{n-1}_I(R))$ and $\mathfrak q \in \Ass D(H^{n-1}_I(R/\mathfrak b))$ by the previous equality. In particular $\mathfrak b \subset \mathfrak q$ since $D(H^{n-1}_I(R/\mathfrak b))$ is annihilated by
$\mathfrak b.$ Therefore $\mathfrak q \in V$ but this is in contradiction to $\mathfrak p \in \Ass D(H^n_I(R)),$ which means $\mathfrak q \in U$ by the Hartshorne-Lichtenbaum vanishing theorem.
We can solve this controversy only in case $x$ is a unit. But this proves the vanishing of both of the local cohomology modules $H^i_I(R), i = n, n-1,$ as required.
\end{proof}

As an application we are able to prove Theorem \ref{0.1} of the Introduction.

\begin{proof} {\bf Theorem \ref{0.1}:} Because $I$ is a complete intersection in $V(I) \setminus \{\mathfrak m\}$ it follows that $IR_{\mathfrak m}$ is generated by $c$ elements in $R_{\mathfrak p}$ for all $V(I) \setminus \{\mathfrak m\}$. That means the conditions (ii), (ii), and (iv) hold for any localization with respect to $\mathfrak p \in V(I) \setminus \{\mathfrak m\}.$ Therefore the equivalence of the conditions (i), (ii), and (iii) follows by virtue of Theorem \ref{4.1}. While the equivalence of (i) and (iv) is a particular case of Theorem \ref{4.2}.

Moreover, if $I$ satisfies one of the equivalent conditions the
conclusion about $H^i_I(R)$ are shown in Corollary \ref{3.8}.
\end{proof}

\section{Examples and Remarks}
Let us discuss the necessity of the local conditions in Theorem \ref{4.1}. By the results of Bass (cf.
\cite{hB}) a local ring is a Gorenstein ring if and only if $\dim_k \Ext^i_R(k, R) = \delta_{n,i},
n = \dim R.$ Moreover, let $(R,\mathfrak m)$ be a Gorenstein ring. Then $R_{\mathfrak p}, \mathfrak p
\in \Spec R,$ is also a Gorenstein ring, i.e. the Gorenstein property localizes. The following example
shows that the property
\[
\dim_k \Ext^i_R(k, H^c_I(R)) = \delta_{d,i}, \quad d = \dim R/I,
\]
does not localize to the corresponding statement for $H^c_{I R_{\mathfrak p}}(R_{\mathfrak p}),
\mathfrak p \in V(I).$

\begin{example} \label{5.1} Let $k$ be an arbitrary field. Let $R = k[|x_0,x_1,x_2,x_3,x_4|]$ denote the
formal power series ring in five variables over $k.$ Let $I =
(x_0,x_1)\cap (x_1,x_2)\cap (x_2,x_3)\cap (x_3,x_4).$ Then $c =
\height I = 2$ and $\dim_k \Ext^i_R(k, H^2_I(R)) = \delta_{3,i}, i
\in \mathbb N.$ Moreover $H^i_I(R) \not= 0$ for all $i \not= 2,3.$
\end{example}

\begin{proof} Obviously we have $c =  2.$ Put
\[
I_1 = (x_0,x_1)\cap (x_1,x_2), \; I_2 = (x_2,x_3)\cap (x_3,x_4), \; J = I_1 + I_2.
\]
By the aid of the Mayer-Vietoris sequence with respect to $I_1$ and $I_2$ it follows that
\[
H^i_I(R) = 0 \text{ for } i \not= 2,3 \text{ and } H^3_I(R) \simeq H^4_J(R).
\]
Moreover $J = J_1 \cap J_2$ with $J_1 = (x_0,x_1,x_3,x_4)$ and $J_2 = (x_1,x_2,x_3),$ as it is easily seen.
A second use of the Mayer-Vietoris sequence with respect to $J_1$ and $J_2$ gives a short exact
sequence
\[
0 \to H^4_{J_1}(R) \to H^4_J(R) \to E \to 0.
\]
Recall that $J_1 + J_2 = \mathfrak m$ and $H^5_{\mathfrak m}(R) \simeq E,$ where $\mathfrak m$ denotes
the maximal ideal of $R$ and $E = E_R(R/\mathfrak m).$ Because of $H^i_{J_1}(R) = 0$ for all
$i \not = 4$ the truncation process provides a short exact sequence
\[
0 \to H^4_{J_1}(R) \to E_R(R/J_1) \to E \to 0.
\]
Localizing both exact sequences at $x_2$ implies the following isomorphisms
\[
H^4_J(R)_{x_2} \simeq H^4_{J_1}(R)_{x_2} \simeq E_R(R/J_1).
\]
Recall that $x_2$ acts bijectively on $E_R(R/J_1).$ Moreover there is the naturally defined
exact sequence
\[
0 \to H^0_{x_2}(H^4_J(R)) \to H^4_J(R) \stackrel{f}{\to} H^4_J(R)_{x_2} \to H^1_{x_2}(H^4_J(R)) \to 0.
\]
In the next step we show that $f$ is an isomorphism. Therefore we have to show that
$H^i_{x_2}(H^4_J(R)) = 0$ for $i = 0,1.$ To this end consider the short exact sequence
\[
0 \to H^1_{x_2}(H^{i-1}_J(R)) \to H^i_{x_2R + J}(R) \to H^0_{x_2}(H^i_J(R)) \to 0
\]
(cf. \cite[Corollary 1.4]{pS3}). Because of the equality $\Rad (x_2R + J) = (x_1,x_2,x_3)R$ it follows that
$H^i_{x_2R + J}(R) = 0$ for all $i \not= 3.$ With this in mind the previous short exact sequence
implies that $f$ is an isomorphism. This means that $H^3_I(R) \simeq H^4_J(R) \simeq E_R(R/J_1)$
and therefore $H^3_I(R)$ is an injective $R$-module.

Finally consider the spectral sequence
\[
E_2^{p,q} = \Ext^p_R(k, H^q_I(R)) \Longrightarrow E^{p+q}_{\infty} =
\Ext_R^{p+q}(k,R).
\]
Because of $\Ext^p_R(k, H^3_I(R)) = 0$ for all $p \in \mathbb Z$ it degenerates to
isomorphisms
\[
\Ext^p_R(k, H^2_I(R)) \simeq \Ext^{p+2}_R(k, R) \text{ for all } p \in \mathbb Z.
\]
Because $R$ is a Gorenstein ring it follows that $\dim_k \Ext^p_R(k, H^2_I(R)) = \delta_{3,p}$
as required.
\end{proof}

Another problem related to our considerations is the finiteness of the Bass numbers of
$H^c_I(R).$ Recall that the Bass numbers of a finitely generated $R$-module are always
finite (cf. \cite{hB}). This is not the case for the Bass numbers of $H^c_I(R), c = \height I.$

\begin{example} \label{5.2} Let $k$ denote a field and $R = k[|x,y,u,v|]/(xu-yv),$ where
$k[|x,y,u,v|]$ denotes the power series ring in four variables over
$k.$ Let $I = (u,v)R.$ Then $\dim R = 3, \dim R/I = 2$ and $c = 1.$
It follows that $H^i_I(R) = 0$ for $i \not= 1,2.$ The truncation
complex with the short exact sequence (cf. \ref{2.1})
\[
0 \to H^c_I(R)[-c] \to \gam (E^{\cdot}) \to C^{\cdot}_R(I) \to 0
\]
induces an injection
\[
0 \to \Hom_R(k, H^2_I(R)) \to \Ext^2_R(k, H^1_I(R)).
\]
Hartshorne (cf. \cite[{\S}3]{rH2}) has shown that the socle of
$H^2_I(R)$ is not a finite dimensional $k$-vector space. Therefore,
the second Bass number of $H^1_I(R)$ is infinite.
\end{example}

As mentioned at the beginning a set-theoretic complete intersection
is a cohomologically complete intersection. The converse is not
true. Let $\mathfrak p \subset k[x_0,x_1,x_2,x_3]$ a homogeneous
prime ideal of dimension two. Then $\mathfrak p$ in $R =
k[x_0,x_1,x_2,x_3]_{(x_0,x_1,x_2,x_3)}$ is always a cohomologically
complete intersection because $H^3_{\mathfrak p}(R) = 0$ by
\cite[Theorem 7.5]{rH3}. In the following we will remark that the property
of being a set-theoretic complete intersection is -- by virtue of
Theorems \ref{4.1} and \ref{4.2} -- also completely
encoded in the local cohomology $H^c_I(R), c = \height I.$

The following result is a particular case of \cite[section 0]{mH2} or, with more details,
in \cite[1.1.4]{mH}. For the sake of completeness we include a proof.

\begin{lemma} \label{5.3} Let $(R, \mathfrak m)$ denote a local
ring. Let $I \subset R$ denote an ideal. Let $f_1,\ldots,f_c, c = \height I,$ be a regular
sequence contained in $I.$ Then the following conditions are equivalent:
\begin{itemize}
\item[(i)] $\Rad I = \Rad (f_1,\ldots, f_c)R.$
\item[(ii)] $I$ is a cohomologically complete intersection and $f_1,\ldots ,f_c$ is a regular
sequence on $\Hom_R(H^c_I(R), E).$
\end{itemize}
\end{lemma}

\begin{proof} We show (ii) $\Longrightarrow$ (i). It is easy to see that $\cd I = 0$ if
and only if $\ara I = 0.$ Now let $f \in I$ denote an $R$-regular element. Then the multiplication
map by $f$ induces an exact sequence
\[
0 \to H^{c-1}_I(R/fR) \to H^c_I(R) \stackrel{f}{\to} H^c_I(R) \to H^c_I(R/fR) \to 0
\]
and the vanishing $H^i_I(R/fR) = 0$ for all $i \not= c-1, c.$ Then $\cd I (R/fR) = c-1$
if and only if $f$ is regular on $\Hom_R(H^c_I(R), E).$ So an induction on $c$ proves the claim.

The converse  (i) $\Longrightarrow$ (ii) follows by a similar consideration.
\end{proof}

Because of the previuos arguments it would be of some intererst to understand the
structure of $H^c_I(R), c = \height I,$ in a better way.

\begin{example} \label{5.4} Let $R = k[x_0,x_1,x_2,x_3]_{(x_0,x_1,x_2,x_3)}$ and let $\mathfrak p$
be the defining ideal of the rational quartic given parametrically by $(s^4,s^3t,st^3,t^4)$ in
$\mathbb P^3_k.$ Therefore
\[
\mathfrak p = (x_0x_3-x_1x_2, x_1^3-x_0^2x_2, x_1^2x_3-x_0x_2^2, x_1x_3^2-x_2^3).
\]
By the above remark $\cd \mathfrak p = 2,$ so that $\mathfrak p$ is
cohomologically a complete intersection. Let $k$ a field of positive
characteristic $p.$ Let $n, r, s \in \mathbb N$ such that $p^n = 3r
+ 4s.$ Then it was shown (cf. \cite[II.2.(ii)]{RS}) that $\mathfrak
p = \Rad (F,G),$ where
\[
F = x_1^3-x_0^2x_2, G = (x_2^4-x_0x_3^3)^{p^n} + 3x_0^rx_1^sx_2^{4r+2s}x_3^{p^n}(x_0^rx_1^sx_3^{p^n}-x_2^{4r+2s}).
\]
Therefore $F, G$ is a regular sequence on $\Hom_R(H^2_{\mathfrak p}(R) , E).$ It is an open problem
whether there is such a regular sequence in the case of $\text{ char } k = 0.$
\end{example}

Related to the characterization of a Gorenstein ring and the results in Theorem \ref{4.2} 
there is the following problem concerning the Bass numbers. 

\begin{problem} \label{5.5}
 A Gorenstein ring $(R,\mathfrak m)$ is a Cohen-Macaulay ring of type 1. That means the following statement about the Bass numbers. Suppose that
$\dim_k \Ext^i_R(k, R) = \delta_{d,i}$ for all $i \leq d = \dim R.$
Then $\dim_k \Ext^i_R(k, R) = \delta_{d,i}$ for all $i \in \mathbb
Z.$ We do not know whether it will be sufficient to replace the
condition (ii) in Theorem \ref{4.2} by
\[
\dim_{k(\mathfrak p)} \Ext^i_{R_{\mathfrak p}}(k(\mathfrak p),
H^c_{IR_{\mathfrak p}}(R_{\mathfrak p})) = \delta_{h(\mathfrak p),i}
\text{ for all } i \leq h(\mathfrak p).
\]
\end{problem}

\end{document}